\documentclass[a4paper,twoside,english,reqno]{amsart} 
\usepackage{latexsym,amsfonts,amssymb,verbatim,mathrsfs,cite,verbatim,amsmath,amsthm}
\usepackage[mathscr]{euscript}
\usepackage{mathptmx,mathtools}
\usepackage{xcolor,babel}
\usepackage{enumerate}
\usepackage{soul,cite,lineno,cancel}
\modulolinenumbers[2] 
\newtheorem{theorem}{Theorem}[section] 
\newtheorem{corollary}[theorem]{Corollary} 
\newtheorem{axiom}[theorem]{Axioms} 
 
\newtheorem{proposition}[theorem]{Proposition}
\newtheorem{assumption}[theorem]{Assumption}
\newtheorem{assumptions}[theorem]{Assumptions}

\newtheorem{definition}[theorem]{Definition}
\newtheorem{definitions}[theorem]{Definitions}
\theoremstyle{definition}

\newtheorem{remarks}[theorem]{Remarks}
\newtheorem{example}[theorem]{Example}

\newenvironment{dedication}
  {
   \itshape             
   \raggedleft          
  }
  {\par 
  }

\usepackage{academicons}
\usepackage{tikz,xcolor,hyperref}

\definecolor{lime}{HTML}{A6CE39}
\DeclareRobustCommand{\orcidicon}{
	\begin{tikzpicture}
	\draw[lime, fill=lime] (0,0) 
	circle [radius=0.16] 
	node[white] {{\fontfamily{qag}\selectfont \tiny ID}};
	\draw[white, fill=white] (-0.0625,0.095) 
	circle [radius=0.007];
	\end{tikzpicture}
	\hspace{-2mm}
}
\foreach \x in {A, ..., Z}{%
	\expandafter\xdef\csname orcid\x\endcsname{\noexpand\href{https://orcid.org/\csname orcidauthor\x\endcsname}{\noexpand\orcidicon}}
}

\newcommand{\Addresses}{{
  \bigskip
  \footnotesize

\textit{Antonio Boccuto, Anna Rita Sambucini}:
 Department of Mathematics and Computer Sciences, 06123 Perugia, (Italy).
 Email: antonio.boccuto@unipg.it,  anna.sambucini@unipg.it, ORCID ID: 0000-0003-3795-8856  \orcidA{}, 0000-0003-0161-8729  \orcidB{}; Anna Rita Sambucini ResearcherID: B-6116-2015.
}}

\begin{document}
\begin{dedication}
  To Professor Paul Leo Butzer,\\
the Scientist and the Man, with deep esteem\\  and 
  admiration  for His
notable contributions in Mathematics.\\
\phantom{A}\\
  \end{dedication}

\title[Some applications of modular convergence...]{Some applications of modular convergence  in vector lattice setting}
\author{Antonio Boccuto,  Anna Rita Sambucini}
\date{}
\maketitle 
\begin{abstract}
The main purpose of this paper is to apply  the theory of
vector lattices and the  related abstract modular convergence  to the 
context of 
Mellin-type kernels and (non)linear vector lattice-valued
operators, following the construction of an integral given in earlier papers.
\end{abstract} 

{\small

}
\maketitle
\section{Introduction}\label{sec1}

In this paper
we continue the study started in 
\cite{BCSVITALI,BCSLp,BS2021,bm1,bmmellin,BDMEDITERRANEAN,bbdmkorovkin}
 were we have  extended to the vector lattice setting the problem of approximating a function $f$ by means of
Urysohn-type integral operators in  the setting of modular convergence.
 These operators are particularly useful in order to approximate a continuous or analog  signal
by means of discrete samples, and therefore they are widely applied for  instance 
 in reconstructing images, see for example \cite{AV2,A3,Bu2016-1,Bu2016-2,CS2018,CSV2019,CGAL2021,CGAL2022}.
Professor Paul Leo Butzer   is
certainly one of the Masters and pioneers  in  Approximation Theory and Signal Analysis and his works have
 formed many generations of researchers. He has worked continuously on these topics, as evidenced by his
  scientific production
 \cite{Bu2014,Bu2016-1,Bu2016-2,Bu2016-3,Bu2003,Bu2007,Bu2010,Bu1973,Bu1986,Bu2011,Bu2000,Bu1992, Bu1969,Bu2008}. 
We met him in 1990, in Capri during the conference ``Quarto Convegno di Analisi Reale e Teoria della Misura",
 in which he  presented the results of \cite{capri} in the main conference
 {\em ``The sampling theorem and its unique role in various branches of mathematics''}, which 
 has influenced our research in this area
and related topics and applications.
In \cite{BS2021} we have considered   vector lattice-valued functions defined
on a metric space, vector lattice-valued measures and we have constructed  a 
vector lattice-valued integral, involving a triple  of vector lattices,
linked by a suitable ``product''  structure. The results obtained are
 extensions of  those on the integrals investigated in \cite{BC2009}, where only finite
measures are considered, and in \cite{BCSVITALI}, where 
the vector lattices considered there have suitable properties. 
In this paper we continue   the  investigation started in \cite{BS2021}, and give examples 
and applications on moment and Mellin-type kernels, as well as linear and nonlinear 
vector lattice-valued operators.

In Section \ref{due} we recall some properties of 
vector lattices and introduce an axiomatic definition of limit and limit superior together
 with the definition of the integral. 
In Section \ref{tre} we recall the main properties   of vector lattice-valued modulars, while in Section 
\ref{ex} we give our structural assumptions  on operators and apply our results to the case of 
 Mellin-type kernels, and operators  which can be linear or nonlinear.

\section{  Axiomatic convergence and integration  in vector lattices}\label{due} 
Let  $(\mathbf{X}, \leq_X)$ be a vector lattice, and let $\vee$, $\wedge$ indicate the 
lattice suprema and infima
 in $\mathbf{X}$. The space
 $\mathbf{X}$ is  \textit{Dedekind complete} if and only if each  $\emptyset \neq A \subset \mathbf{X}$ which is 
order bounded from above, has a lattice supremum  in $\mathbf{X}$, called $\bigvee A$.
In this paper $\mathbf{X}$ will be a Dedekind complete vector lattice, and  $\mathbf{X}^{+}$ will be its positive cone. 
For any $x \in \mathbf{X}$, put $\lvert x \rvert  = x \vee (-x)$.
An extra element $+\infty$ is added to $\mathbf{X}$, 
which extends in a natural way the operations and the structure of order. The symbol
$\overline{\mathbf{X}}$ denotes the set $\mathbf{X} \cup \{+ \infty\}$, 
and we suppose $0 \cdot (+ \infty) = 0$, by convention.
Let ${\mathcal T}:= \{ (x_n)_n \subset \mathbf{X}\}$  and ${\mathcal T}^{+}=\{ (x_n)_n \in {\mathcal T}$:
$x_n \geq_X 0, \,\, \forall \,\, n \in \mathbb{N}\}$.\\
An \emph{$(o)$-sequence} $(\sigma_l)_{l \in \mathbb{N}}$ in  $\mathbf{X}^{+}$ is a decreasing sequence such that
$\wedge_l \sigma_l=0$. For what  unexplained  we refer, for 
instance,  to \cite{BC2009,BS2021,LZ,MEYER,SCHAEFER,VULIKH,KOZLOWSKI}.\\
Now we recall the axioms given in \cite{BS2021} in order to present the 
abstract convergence in the vector lattice context (see, 
e.g., \cite[Definition 2.1]{BC2009}  
and \cite{bbdmkorovkin}). 
\begin{axiom}\label{convergenze}
{\rm (\cite[Axioms 2.1]{BS2021})}
\rm Given  a linear subspace $\mathcal{S}$ of ${\mathcal T}$,
a \em convergence \rm is a pair $(\mathcal{S},\ell)$, with
 $\ell:\mathcal{S} \to \mathbf{X}$,  which satisfies the conditions 
of linearity and monotonicity  and the
following properties for every $(x_n)_n$, $(y_n)_n , (z_n)_n\in \mathcal{S}$: 
\begin{itemize} \label{convergenceaxioms}
\item[{\ref{convergenze}.a)}] 
If $(x_n)_n$ has the property 
that $x_n=l$ definitely, then $(x_n)_n \in \mathcal{S}$ and $\ell((x_n)_n)=l$; if 
 the set $\{n \in \mathbb{N}: x_n \neq y_n \}$ is finite and $(x_n)_n \in \mathcal{S}$, then $(y_n)_n \in \mathcal{S}$ 
and $\ell((y_n)_n)=\ell((x_n)_n)$.
\item[{\ref{convergenze}.b)}] 
If $(x_n)_n \in \mathcal{S}$, then $(\lvert x_n \rvert)_n \in \mathcal{S}$ and $\ell((\lvert x_n \rvert)_n)= \vert \ell((x_n)_n) \rvert$.
\item[{\ref{convergenze}.c)}]
If  $(x_n)_n$, $(z_n)_n \in \mathcal{S}$, 
$\ell((x_n)_n)=\ell((z_n)_n)$, and  $x_n \leq_X y_n \leq_X z_n$ definitely, then $(y_n)_n \in \mathcal{S}$ and
$\ell((x_n)_n)= \ell((y_n)_n)=\ell((z_n)_n))$.
\item[{\ref{convergenze}.d)}]
 If $u \in \mathbf{X}^{+}$, then 
$\Bigl(\dfrac1n u\Bigr)_n \in$ 
$\mathcal{S}$ and 
$\ell\Bigl(\Bigl(\dfrac1n u\Bigr)_n\Bigr)=0$.
\end{itemize}
\end{axiom}
\noindent
An axiomatic approach for a ``limit superior''-type vector lattice-valued operator  
 satisfying Axioms \ref{convergenceaxioms} is the following:

\begin{axiom}\label{limsuppresentation}
{\rm (see \cite[Axioms 2.2]{BS2021})}
\rm Let ${\mathcal T}$, $\mathcal{S}$ be as in Axioms \ref{convergenze},  
$(x_n)_n, (y_n)_n \in {\mathcal T}^{+}$,
 and define a function $\overline{\ell}: {\mathcal T}^{+}
\to \overline{\mathbf{X}}^{+}$  satisfying the 
 conditions of subadditivity and monotonicity, and the next properties:
\begin{itemize}
\item[{\ref{limsuppresentation}.a)}]
If   $x_n=y_n$ definitely, then $\overline{\ell}((x_n)_n)=\overline{\ell}((y_n)_n)$.
\item[{\ref{limsuppresentation}.b)}]
If $(x_n)_n \in  \mathcal{S}$, then   ${\overline{\ell}}((x_n)_n) =\ell((x_n)_n)$.
\item[{\ref{limsuppresentation}.c)}]
If  $(x_n)_n$ has the property that  ${\overline{\ell}}((x_n)_n) =0$, then $(x_n)_n \in \mathcal{S}$ and \mbox{$\ell((x_n)_n)=0$.}
\end{itemize}
\end{axiom}
In order to introduce an abstract integral for vector lattice-valued 
functions with respect to (possibly infinite) vector lattice-valued measures, 
extending the integrals given in \cite{BC2009,BC20092,BCSVITALI}, we assume
some relations between the order and the addition and product operations.

\begin{assumption}\label{ass}
{\rm (\cite[Axioms 3.1]{BS2021})}
\rm Let $\ell$,
$\ell_{\mathbb{R}}$ be two convergences,
satisfying Axioms \ref{convergenze}. 
We say that  $(\mathbf{X},\mathbb{R},\mathbf{X})$ is a \emph{product triple}
iff there   is 
a ``product operation'' $\cdot :\mathbf{X} \times \mathbb{R} \to \mathbf{X}$, 
fulfilling the distributivity laws with respect to the sum,
compatible with respect to the order and to the operations 
of supremum and infimum, and satisfying 
the following two additional  conditions:
\begin{itemize}\label{compatibility}
\item[{\ref{ass}.1)}] 
if $(x_n)_n \subset {\mathbf{X}}$ with $\ell((x_n)_n) = 0$ and $y \in {\mathbb{R}}$, then 
${\ell}((x_n \cdot y )_n)=0$;
\item[{\ref{ass}.2)}] 
if $x \in {\mathbf{X}}$ and $(y_n)_n\subset {\mathbb{R}}$ with $\ell_{\mathbb{R}} ((y_n)_n) = 0$,
then ${\ell}((x \cdot y_n )_n)=0$.
\end{itemize}
\end{assumption}
Let $G=(G,d)$ be a metric space,  ${\mathcal P}(G)$ be the family of all subsets of $G$, ${\mathcal A} \subset {\mathcal P}(G)$ be an algebra,
$\mu:{\mathcal A} \to  \overline{\mathbb{R}}^{+}$ be a finitely additive measure; $\mu$ is said to be \emph{$\sigma$-finite} if and only if there is an (increasing) sequence  $(B_n)_n$ from ${\mathcal A}$, such that  
$\mu(B_n) \in {\mathbb{R}}^{+}  \text{ for   all   } n \in \mathbb{N}$ and
$\displaystyle{\bigcup_{n\in \mathbb{N}} B_n=G}$.\\

Now we introduce an integral for $\mathbf{X}$-valued functions with respect to a  positive,  finitely additive and $\sigma$-finite 
extended real-valued  
measure $\mu$, associated with 
convergences $\ell$, $\ell_{\mathbb{R}}$, 
fulfilling Axioms \ref{convergenze}. 
  The integral  will be an element of $\mathbf{X}$.
Similar constructions  were given also  in  
\cite{BC2009,BC20092,BCSVITALI},
in some particular cases. \\ 

We denote by $\mathscr{S}$
the set  of all simple functions defined on $G$, which assume
a finite number of values and with support contained in a set of finite measure $\mu$. For such functions, the integral is defined in a natural way (see, e.g., 
\cite{BC2009, BS2021}). \\

Now, to introduce the integral for more general functions, we use 
 the definitions of uniform convergence and convergence in measure 
and in $L^1$.
\begin{definitions}\label{measureuniform} \rm 
\mbox{(see \cite[Definition 3.2]{BS2021})}
Let $A \in {\mathcal A}$ and $(f_n)_n \subset \mathbf{X}^G$.
We say that: 
\begin{itemize}
\item $(f_n)_n$ is \emph{uniformly convergent} to $f\in \mathbf{X}^G$ on $A$
if and only if   
\[\displaystyle{\ell \Bigl(\Bigl( \bigvee_{g \in A} \, \lvert f_n(g)-f(g) \rvert \Bigr)_n\Bigr)=0};\]
\item 
$(f_n)_n$ \em converges in $\mu$-measure \rm to $f\in \mathbf{X}^G$ on $A$ if and only if  there exists
 $(A_n)_n \subset \mathcal{A}$, such that $\ell_{\mathbb{R}}((\mu(A \cap A_n))_n)=0$ and 
\[\displaystyle{ \ell \Bigl(\Bigl(\bigvee_{g \in A \setminus A_n} \lvert f_n(g)-f(g) \rvert\Bigr)_n\Bigr)=0};\]
\item  $(f_n)_n \subset \mathscr{S}$ \em converges in $L^1$ \rm to $f \in \mathscr{S}$ if and only if 
\[
 \ell \left( \left(\int_G \lvert f_n(g)-f(g) \rvert \, d\mu(g) \right)_n \, \right)=0. 
\]
\end{itemize}
\end{definitions} 
\noindent Convergence in measure is a consequence of the uniform convergence  and, if ${\mathbf{X}}=\mathbb{R}$,  the convergences here defined 
are equivalent to the classical ones (see, e.g., \cite[Remark 3.3]{BC2009}).
\begin{definition}\label{eac}
\mbox{\rm (see \cite[Definition 3.5]{BS2021})}
\rm A sequence $(f_n)_n \subset \mathscr{S}$
is $\mu$-\emph{equiabsolutely continuous} if and only if
\begin{itemize}
\item[{\rm \ref{eac}.1)}] 
$\displaystyle{\ell \left( \left(\int_{A_n}\, \lvert f_n (g) \rvert \, d\mu(g)\right)_n \right)=0}$\,\,  whenever \,\, $\ell_{\mathbb{R}}((\mu( A_n))_n)=0$;
\item[{\rm \ref{eac}.2)}]
 there exists an increasing sequence  $(B_m)_m \subset {\mathcal A}$ with
$\mu(B_m)\in {\mathbb{R}} < + \infty$ for each $m \in \mathbb{N}$, and
\begin{eqnarray*}\label{eaceac2}
\ell\Bigl(\Bigl(\overline{\ell}\Bigl( \Bigl(\int_{G \setminus B_m} \lvert f_n(g) \rvert \, d\mu(g) \Bigr)_n   \Bigr) \Bigr)_m \Bigr)=0.
\end{eqnarray*}
\end{itemize}
\end{definition}
We say that a sequence $(f_n)_n$ in $\mathscr{S}$ is \emph{defining for $f \in {\mathbf{X}}^G$} if and only if it converges in $\mu$-measure to $f$  on each set 
$A \in {\mathcal A}$ of finite measure $\mu$,
and their  integrals are $\mu$-equiabsolutely continuous.
\begin{definition}\label{integrabilita} \rm 
A positive function $f \in {\mathbf{X}}^G$ is \em integrable \rm on $G$ if and only if there are a defining sequence $(f_n)_n$ 
for $f$ and   a function 
$l:{\mathcal A} \to \mathbf{X}$, such that
\begin{eqnarray*}
	\ell \Bigl( \Bigl(\bigvee_{A\in {\mathcal A}} \lvert \int_A f_n(g) \, d\mu(g) - l(A)\rvert \Bigr)_n \Bigr)=0,
\end{eqnarray*} 
and we put
\begin{eqnarray}\label{integral}
	\int_A f(g) \, d\mu(g):=l(A) \quad \text{ for  every  } A \in {\mathcal A}.
\end{eqnarray}
\end{definition}
For not necessarily positive functions the integral, as usual, is given using $f^+(g)=f(g)\vee 0$, $f^-(g)=(-f(g))\vee 0$,\, for every $g \in G$, namely:
\begin{definition}\label{integrabilitaa} \rm
A function $f: G \to {\mathbf{X}}$ is \emph{integrable} on $G$ if
and only if the functions $f^{\pm}$  are integrable on $G$, 
and in this case we define 
\begin{eqnarray}\label{integrall}
\int_A f(g) \, d\mu(g) = \int_A f^+(g) \, d\mu(g) - \int_A f^-(g) \, d\mu(g),\quad A \in {\mathcal A}.
\end{eqnarray}
\end{definition}
For the properties of this integral and the convergence results, we refer to \cite[Section 3]{BS2021}; in particular, when 
${\mathbb{X}}=\mathbb{R}$ and $\ell_{\mathbb{R}}$
is the usual limit, this integral coincides with the Lebesgue one.
We denote by $\mathscr{L}({\mathbf{X}}, \mu)$ the space of all functions 
$f: G \to \mathbf{X}$ which are integrable
 in the sense of formulas 
(\ref{integral}) and (\ref{integrall}).
\vspace{3mm}

\noindent Now we recall uniform continuity in the vector lattice context
(see, e.g., \cite{BCSVITALI, BS2021}).
\begin{itemize}
\item We say that 
$f:G \to \mathbf{X}$ is $\emph{uniformly continuous on }$  $G$ if and only if $\,\, \exists\, \,\,  u \in {\mathbf{X}}^ +: $ 
$\forall \,\, \varepsilon \in \mathbb{R}^+,  \,\,\exists\,\, \delta \in \mathbb{R}^+ $ such that
\[ \lvert f(g_1) -f(g_2)\rvert \leq_{X}  \varepsilon \, u 
\quad \forall \,\, g_1, g_2\in G\, : \,\,
d(g_1,g_2) \leq \delta.\]
\item 
A function $\psi:\mathbf{X}\to \mathbf{X}$ is \emph{uniformly continuous on }  $\mathbf{X}$ if and only if 
 $\forall \,\, u \in {\mathbf{X}}^+$ and $ \varepsilon \in  \mathbb{R}^+ $,
 $\exists\,\, w \in {\mathbf{X}}^+$ and $
\exists\,\,  \delta \in  \mathbb{R}^+ $ such that
\[\lvert \psi(x_1)-\psi(x_2) \rvert \leq_{X} \, \varepsilon \, w ,\, \,\mbox{  whenever  } \, \lvert x_1-x_2\rvert \leq_{X} \delta \, u.\]
\end{itemize}
If $G={\mathbf{X}}=\mathbb{R}$ with the usual topology, the two  definitions of uniform continuity coincide with the classical one
(see  \cite{BCSVITALI, BS2021}).
\section{Modulars}\label{tre}
 Let $T$ be a vector subspace of ${\mathbf{X}}^G$ such that,
if $f \in T$ and $A \in {\mathcal A}$, then ${\color{red}
\lvert f \rvert }\in T$ and
$f \cdot \chi_A \in T$, 
where the symbol $\chi_A$ denotes the 
\emph{characteristic 
function} associated with the set $A$, that is the function which
associates the value $1$ to every element of $A$ and the value $0$ 
to every element of $G \setminus A$. 
\\

A functional $\rho:T \to \overline{\mathbf X}^{+}$ is a \textit{modular} on $T$ if and only if the following properties hold:
\begin{itemize}
\item[($m_0$)] $\rho(0)=0$;
\item[($m_1$)] $\rho(-f)=\rho(f)$ for each $f \in T$;
\item[($m_2$)] $\rho(c_1 f + c_2 h) \leq_{X} \rho(f) + \rho(h)$ for 
every $f$, $h \in T$ and for all $c_1$, $c_2 \in 
\mathbb{R}_0^{+}$ such that
\mbox{$c_1 + c_2 =1$.}
\end{itemize}
A modular $\rho$ is \textit{monotone} if and only if $\rho(f) \leq_{X} \rho(h)$ for every $f$, $h \in T$ such that $
\lvert f \rvert \leq_{X} \lvert h \rvert$.
If $f \in T$, then $
\lvert f \rvert
 \in T$ and $\rho(f) = 
\rho(\lvert f \rvert)$ (see, e.g., \cite{BMV, BCSVITALI}); while $\rho$ is \emph{convex} if and only if  $\rho(c_1 f + c_2 h) \leq_{X} c_1 \, \rho(f) + c_2 \, \rho(h)$ 
for all $f$, $h \in T$, $c_1, c_2 \in 
\mathbb{R}_0^{+}$ with $c_1 + c_2 =1$.
For a related literature on modulars see, e.g., \cite{BMV, BCSVITALI,
BCSLp, KOZLOWSKI} and the 
 references therein.\\
Proceeding now  analogously as in \cite[Proposition 3.1]{BCSVITALI}, it is possible to see that, 
if \mbox{$\varphi: {\mathbf{X}} \to {\mathbf{X}}$} is increasing on $
{\mathbf{X}}^{+}$ and $\varphi(0)=0$,
then it makes sense to define the set
\begin{eqnarray*}\label{preorlicz}
	{ \mathcal L}^{\varphi}= \Bigl\{ f \in  {{\mathbf{X}}}^G: 
	\, \, \int_G \varphi( \lvert f(g) \rvert ) \, d\mu(g) \text{   exists   in  } {\mathbf{X}} \Bigr\} ,
\end{eqnarray*}
and then the $\mathbf{X}$-valued operator $\rho^{\varphi}$  defined by
\begin{eqnarray*}\label{orliczmodular}
	\rho^{\varphi}(f)=\int_G \varphi(\lvert f(g) \rvert) \, d\mu(g),
\quad f \in {\mathcal L}^{\varphi},
\end{eqnarray*} 
is a monotone modular  and, if $\varphi$ is convex, then  $\rho^{\varphi}$ is convex on the set of the positive functions 
of ${\mathcal L}^{\varphi}$. 
 The set 
\begin{eqnarray*}\label{orliczriesz} 
	\displaystyle{L^{\varphi}(G)= \Bigl\{f \in {\mathbf{X}}^G: 
\bigwedge_{a \in \mathbb{R}^+}  \rho^{\varphi}(a \, f) =0  
\Bigr\}}
\end{eqnarray*}
is the \emph{Orlicz space} generated by $\varphi$. 
Note that $L^{\varphi}(G)$ is  a vector space
and,  if $a \in \mathbb{R}^+$ and $a \, f \in L^{\varphi}(G)$, then $b \, f \in
L^{\varphi}(G)$ whenever $b \in ]0,a[$.
The subspace $E^{\varphi}(G)$ of $L^{\varphi}(G)$, defined by
setting
\begin{eqnarray*}\label{Evarphi}
	E^{\varphi}(G) :=\{ f \in L^{\varphi}(G): \rho^{\varphi}
	(a \, f) \in \mathbf{X} \text{  for  each   }a \in
	 \mathbb{R}^+ \}, 
\end{eqnarray*}
is the \emph{space of the finite elements of} $L^{\varphi}(G)$.

We say that a sequence $(f_n)_n$  from $L^{\varphi}(G)$ is \textit{modularly convergent} to $f \in L^{\varphi}(G)$ if and only if 
there exists $a \in \mathbb{R}^+ $ with
\begin{eqnarray}\label{modularconvergence0}
\displaystyle{{\ell} (( \rho^{\varphi}(a(f_n-f)))_n)=0}. 
\end{eqnarray}

\section{Structural assumptions on operators and examples}\label{ex}
As examples and applications of the results given in the previous 
sections, we consider (non)linear
Mellin-type operators with values in vector lattices.
In this setting, we 
take $G=(\mathbb{R}^+,d_{\ln})$, where
$d_{\ln} (t_1, t_2) =\lvert \ln t_1 - \ln t_2\rvert$, $t_1, t_2 \in \mathbb{R}^+ $,
and for each measurable set $S \subset \mathbb{R}^+$ we 
 put
$\displaystyle{\mu(S)= \int_S \frac{dt}{t}}.$
Let $\mathcal M$ be the set of all sequences of functions
$\widetilde{L}_n$ defined on $\mathbb{R}^+$,
non-negative and $\mu$-integrable,
and such 
that for any $n \in \mathbb{N}$ and 
every $s \in \mathbb{R}^+$ (resp., 
$t \in \mathbb{R}^+$) the map $t \mapsto \widetilde{L}_n \Bigl( 
\dfrac{t}{s} \Bigr)$ (resp., $s \mapsto \widetilde{L}_n \Bigl( 
\dfrac{t}{s} \Bigr)$\ )
is ($\mu$-integrable and) bounded. 
We now introduce some structural assumptions.
\begin{assumptions} \label{assumptionsmellin}\rm
~
\begin{itemize}
\item[\ref{assumptionsmellin}.a)] Let $\Psi$ be the class of all functions $\psi :\mathbf{X}^{+} \to  \mathbf{X}^{+}$ 
with the property that
\begin{itemize}
\item[{\ref{assumptionsmellin}.a.1)}] 
$\psi$ is uniformly continuous and increasing on ${\mathbf{X}}^{+}$, $\psi(0)=0$ and 
$\psi(v) \in {\mathbf{X}}^+\setminus \{0\}$ for all $v \in {\mathbf{X}}^+\setminus \{0\}$. 
\end{itemize}
Let $\Xi= (\psi_n)_n \subset \Psi$ be a sequence of functions, satisfying the following conditions:
\begin{itemize}
\item[{\ref{assumptionsmellin}.a.2)}] 
$(\psi_n)_n$ is \emph{equicontinuous at} 0, namely for each
$u \in {\mathbf{X}}^+\setminus \{0\}$ and $\varepsilon \in  \mathbb{R}^+$
there exist $w \in {\mathbf{X}}^+\setminus \{0\}$ and  
 $\delta \in \mathbb{R}^+$ such that $\psi_n(x) \leq_{X}   \varepsilon \, w$ whenever $x \leq_{X} \delta \, u$ and for all
$n \in \mathbb{N}$;
\item[{\ref{assumptionsmellin}.a.3)}] 
for each $v \in {\mathbf{X}}^{+}$ the sequence $(\psi_n(v))_n$ is \emph{order equibounded}, that is there is
$A_v \in {\mathbf{X}}^+\setminus \{0\}$ 
such that $\psi_n(v) \leq_{X} A_v$ for all $n \in \mathbb{N}$.
\end{itemize}
\item[\ref{assumptionsmellin}.b)] Let $\Xi= (\psi_n)_n \subset \Psi$ be as in \ref{assumptionsmellin}.a), and
denote by ${\widetilde{\mathcal K}}_{\Xi}$ the set of all sequences of functions
$\widetilde{K}_n: \mathbb{R}^+ \times \mathbf{X} \to 
{\mathbf{X}}$, $n \in \mathbb{N}$, such that
\begin{itemize}
\item[\ref{assumptionsmellin}.b.1)]
 $\widetilde{K}_n(\cdot,u) \in \mathscr{L}(\mathbf{X},\mu)$ for each $
u \in {\mathbf{X}}$ and $n \in \mathbb{N}$, and $\widetilde{K}_n(t,0) = 0 $ for all
$n \in \mathbb{N}$ and $t \in \mathbb{R}^+$;
\item[\ref{assumptionsmellin}.b.2)]
there are sequences $(\widetilde{L}_n)_n \subset \mathcal M$ and
$(\psi_n)_n \subset \Psi$ with
\begin{align*}\label{mellinlipschitz}
\lvert\widetilde{K}_n(t,u)-\widetilde{K}_n(t,v)\rvert \leq_{X} \widetilde{L}_n(t) \, \psi_n(\lvert u-v \rvert)
\end{align*}
for all $n \in \mathbb{N}$, $t \in \mathbb{R}^+$ and 
$u$, $v \in \mathbf{X}$.
\end{itemize}
\end{itemize}
\end{assumptions}
Let $\widetilde{\mathbb{K}}=(\widetilde{K}_n)_n \in {\widetilde{\mathcal K}}_{\Xi}$, and consider
a sequence $\widetilde{\textbf{T}}=(\widetilde{T_n})_n$ of nonlinear 
Mellin-type operators defined by
\begin{eqnarray}\label{mellinoperatori}
(\widetilde{T_n} f)(s)=\int_0^{+ \infty} \, \widetilde{K}_n\Bigl(\frac{t}{s},f(t)\Bigr)
\frac{dt}{t}, \quad 
\end{eqnarray}
$n \in \mathbb{N}$,
$s \in \mathcal{\mathbb{R}}^+$, $f \in $ Dom $\widetilde{\textbf{T}}=$ $\displaystyle{\bigcap_{n=1}^{\infty}}$
Dom $\widetilde{T_n}$, where  Dom $\widetilde{T_n}$ is the
set of the functions $f$ for which the integral in 
(\ref{mellinoperatori}) makes sense.
The concept of singularity in the context of Mellin operators can be introduced in the following form, 
given in \cite[Definition 6.7]{BS2021}.
\begin{definition}\label{intro2mellin}
 \rm 
We say that $\widetilde{\mathbb{K}}= (\widetilde{K}_n)_n$ is 
\emph{$U$-singular} if and only if 
there exist 
an infinite set 
$H \subset \mathbb{N}$ and a positive real number
$D^{(1)}$, fulfilling the following conditions:
\begin{itemize}
\item[{\ref{intro2mellin}.1)}] for each $n \in H$, 
$\displaystyle{
\int_0^{+ \infty} \widetilde{L}_n(t) \, \dfrac{dt}{t} \leq D^{(1)}}$;
\item[{\ref{intro2mellin}.2)}]
for every  $(a_n)_n \subset \mathbb{R}$, 
$\overline{\ell}((a_n)_{n\in \mathbb{N}})=
\overline{\ell}((a_n)_{n\in H})$;
\item[{\ref{intro2mellin}.3)}]
$\displaystyle{\int_0^{+ \infty} \widetilde{L}_n(t) \, \dfrac{dt}{t} >0 \quad \text{for  each } n \in \mathbb{N}};$
\item[{\ref{intro2mellin}.4)}]
for any 
$\delta \in \mathbb{R}^+$, $\delta > 1$, one has 
\begin{eqnarray*}\label{checkmellin}
\displaystyle{ {\ell}_{\mathbb{R}} \Bigl( \Bigl(
 \int_{\mathbb{R}^+ \setminus [1/\delta, \delta]}
 \widetilde{L}_n(t) \, \dfrac{dt}{t}
}\Bigr)_n \Bigr)=0;   
\end{eqnarray*} 
\item[{\ref{intro2mellin}.5)}] 
there exist \mbox{$z \in {\mathbf{X}}^+ \setminus \{0\}$} 
and an $(o)$-sequence $(\varepsilon_n)_n$ in ${\mathbb{R}}^+$ 
such that
\begin{eqnarray*}\label{Thetanmellin}
\bigvee_{u \in \mathbf{X} \setminus \{0\}} 
\bigg\lvert \int_0^{+\infty} \widetilde{K}_n(t,u) \, \frac{dt}{t} -u \bigg\rvert
\leq_{X} \varepsilon_n \, z \quad  \text{for   any   } n \in H.
\end{eqnarray*}
\end{itemize}
\end{definition}
%

In general, in the examples and applications, it is not very easy to verify
\ref{intro2mellin}.5) directly.
So, we give a sufficient condition to have {\ref{intro2mellin}.5), 
which is easier to handle and verify in the practice, in the
vector lattice setting.

\begin{proposition}\label{intro3mellin}
Suppose that $\widetilde{K}_n(t,u)=\widetilde{L}_n(t) \, 
\Upsilon_n(u)$, where $\Upsilon_n: \mathbf{X} \to 
\mathbf{X}$, 
\begin{itemize}
\item[{\ref{intro3mellin}.1)}] 
$\displaystyle{\int_0^{+ \infty} \widetilde{L}_n(t) \, \dfrac{dt}{t} =1 \quad \text{for  all } n \in H}$,
and
\item[{\ref{intro3mellin}.2)}] 
there are \mbox{$v \in {\mathbf{X}}^+ \setminus \{0\}$} 
and an $(o)$-sequence $(\sigma_n)_n$ in ${\mathbb{R}}^+$ 
such that
\begin{eqnarray*}\label{Phinmellin}
\bigvee_{u \in \mathbf{X} \setminus \{0\}} 
\lvert \Upsilon_n(u) -u \rvert
\leq_{X} \sigma_n \, v \quad  \text{for   any   } n \in H.
\end{eqnarray*}
\end{itemize}
Then, \ref{intro2mellin}.5) holds.
\end{proposition}
\begin{proof}
Let $v$ and $(\sigma_n)_n$ be related with
\ref{intro3mellin}.2).
One has
\begin{eqnarray*}\label{Bumellin} & &
\bigvee_{u \in \mathbf{X} \setminus \{0\}} 
\bigg\lvert \int_0^{+\infty} \widetilde{K}_n(t,u) \, \frac{dt}{t} -u \bigg\rvert= 
\bigvee_{u \in \mathbf{X} \setminus \{0\}} 
 \bigg\lvert \int_0^{+\infty} \widetilde{L}_n(t) \, \Upsilon_n(u) \, \frac{dt}{t} -u  \bigg\rvert = \\
 &=& \bigvee_{u \in \mathbf{X} \setminus \{0\}} 
\bigg \lvert \Bigl(\int_0^{+\infty} \widetilde{L}_n(t) \,  \frac{dt}{t} \Bigr) \, \Upsilon_n(u)- \Bigl(  \int_0^{+\infty} \widetilde{L}_n(t) \, \frac{dt}{t}\Bigr) \,
u  \bigg\rvert = \\ &=& \bigvee_{u \in \mathbf{X} \setminus \{0\}} 
\, \lvert\Upsilon_n(u)-u \rvert \leq_{X} \, \sigma_n \, v,
\end{eqnarray*}
getting \ref{intro2mellin}.5).
\end{proof}

 For the convenience of the reader we recall  also the following  result, which is a particular case of \cite[Theorems 6.5 and 6.9]{BS2021},
 that will be used in the sequel. For a more detailed description we refer to \cite[Section 6]{BS2021} (https://arxiv.org/pdf/2112.12085.pdf ).
\begin{theorem}\label{unione6.5e6.9}
Under Assumptions {\rm \ref{assumptionsmellin}}, 
assume that $\widetilde{\mathbb{K}}$ is $U$-singular.
Then, for every $f \in {\mathcal C}_c(\mathbb{R}^+)$,
the sequence $(\widetilde{T_n} f )_n$ 
is uniformly convergent to $f$  
on ${\mathbb{R}}^+$, and modularly convergent 
to $f$ with respect to the modular $\rho^{\varphi}$, where the constant $a$ in 
(\ref{modularconvergence0}) can be chosen
independently of $f$.
\end{theorem} 
Now,
 we are ready to give some examples of functions $\widetilde{L}_n$ 
and vector lattice-valued maps $\Upsilon_n$,
such that 
the functions 
\begin{eqnarray*}\label{kappakappa}
\widetilde{K}_n(t,u)=\widetilde{L}_n(t) \, 
\Upsilon_n(u), \quad 
n \in \mathbb{N}, \, \, t \in \mathbb{R}^+, \, \, u \in {\mathbf{X}}
\end{eqnarray*}
satisfy all conditions on
$U$-singularity and Assumptions \ref{intro2mellin},
and such that
\begin{eqnarray}\label{q}
\int_a^b \widetilde{L}_n \Bigl( \dfrac{t}{\cdot} \Bigr) \dfrac{dt}{t} \in L^{q}(\mathbb{R}^+)
\end{eqnarray}
for every $[a,b] \subset \mathbb{R}^+$,
$q \geq 1$ and $n$ large enough, depending on $q$
(see also \cite[Corollary 5.4]{BDMEDITERRANEAN} for kernels, when $\Upsilon_n(u)= u$ and $\mathbf{X} = \mathbb{R}$).
Moreover, from $U$-singularity and using \cite[Remark 6.8.b%
]{BS2021}, we get that
 for any compact 
subset $C \subset \mathbb{R}^+$ 
there is a  set
$B \subset \mathbb{R}^+$ of finite $\mu$-measure such that  
\begin{eqnarray}\label{aamu1}
\bar{\ell}_{\mathbb{R}}  \Bigl( \Bigl( \sup_{t \in C}\int_{{\mathbb{R}}^+ \setminus B} \widetilde{L}_n
\Bigl(\dfrac{s}{t}\Bigr) \, \frac{ds}{s} \Bigr)_n \Bigr)=0.
\end{eqnarray} 
So by Theorem \ref{unione6.5e6.9}
we obtain that the sequence $(\widetilde{T_n} f )_n$ converges 
uniformly to $f$, and
we find a positive real number $a$ such that
${\ell}((\rho^{\varphi}(a(\widetilde{T_n} f-f))_n))=0$ for each 
$f \in {\mathcal C}_c(\mathbb{R}^+)$.
\begin{example}\label{exs1}
 \emph{Moment kernel}.  \,   This kernel is defined by setting
$ \widetilde{L}_n(t) = n \, t^n  \chi_{]0,1[} (t)$, \,\, 
$n \in \mathbb{N}$, $t \in \mathbb{R}^+$.
In \cite[Section 5]{BDMEDITERRANEAN} it is 
shown that
$\displaystyle{\int_0^{+ \infty} \widetilde{L}_n(t) \, \dfrac{dt}{t} =1 }$
(getting \ref{intro3mellin}.1)), and 
\[
\int_{\mathbb{R}^+ \setminus [1/\delta, \delta]}
 \widetilde{L}_n(t) \, \dfrac{dt}{t} =\dfrac{1}{\delta^n} \leq 
\dfrac{1}{n}
\]
for each $n \in \mathbb{N}$ and $\delta \in ]1, + \infty[$. 
From this and Axioms \ref{convergenze}.c), \ref{convergenze}.d)
we obtain \ref{intro2mellin}.4).
Moreover
$0 \leq \widetilde{L}_n(t) \leq n$ for every $n \in \mathbb{N}$
and $t \in \mathbb{R}^+$. From this, since 
 $\widetilde{L}_n \in L^1(\mathbb{R}^+,\mu)$
for every $n \in \mathbb{N}$,
it follows that $(\widetilde{L}_n)_n \subset \mathcal M$.
 Relation
(\ref{q}) follows from the fact that,
for each 
$[a,b] \subset \mathbb{R}^+$, 
for every $n \in \mathbb{N}$ and $s \in \mathbb{R}^+$ it is
\begin{eqnarray*}\label{moment}
\int_a^b \, \widetilde{L_n} \Bigl( \frac{t}{s} \Bigr) \frac{dt}{t} 
=\left\{
\begin{array}{ll}
\dfrac{b^n-a^n}{s^n} &  \text{ if } s \geq b,\\
\dfrac{s^n-a^n}{s^n} &  \text{ if } a\leq s<b,\\
0 &  \text{ if } 0<s<a \end{array} \right.
\end{eqnarray*}
(see also \cite[formula (5.20)]{BDMEDITERRANEAN}).
\begin{figure}[h!]
\begin{center}
\includegraphics[scale=.3]{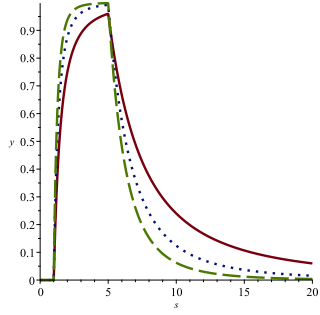}
\end{center}
\caption{\small Moment kernel: $a=1$, $b=5$, $n=2$ (line), $n=3$ (point), $n=4$  (dash).}
\end{figure}
\end{example}

\begin{example}\label{exs2}
 \emph{Mellin-Gauss-Weierstrass kernel}. 
 \,  Let
\[\widetilde{L}_n(t) = \frac{n}{2\sqrt{\pi}} e^{-\frac{n^2}{4} \ln^2 t}, \quad n \in \mathbb{N}, \, \, t \in \mathbb{R}^+.\] 
In \cite[formula (5.17)]{BDMEDITERRANEAN} it is proved that  
$\displaystyle{\int_0^{+ \infty} \widetilde{L}_n(t) \, \dfrac{dt}{t} =1 }$ for any $n \in \mathbb{N}$,
 and that for every
$\delta > 1$ there exists a positive integer $n_0=n_0(\delta)$ 
such that 
\[
\int_{\mathbb{R}^+ \setminus [1/\delta, \delta]}
 \widetilde{L}_n(t) \, \dfrac{dt}{t}  \leq 
\dfrac{2}{\sqrt{\pi}} \, e^{-\frac{n \, \ln \delta}{2}}
\leq \dfrac{1}{n}
\]
whenever $n \geq n_0$. From this and Axioms \ref{convergenze}.c), \ref{convergenze}.d)
we get \ref{intro2mellin}.4).
\\
As before,
$0 \leq \widetilde{L}_n(t) \leq n$ for any $n \in \mathbb{N}$
and $t \in \mathbb{R}^+$. From this and the 
$\mu$-integrability
 of the $\widetilde{L}_n$,
we obtain that $(\widetilde{L}_n)_n \subset \mathcal M$.

\begin{figure}[h!]
\begin{center}
\includegraphics[scale=.3]{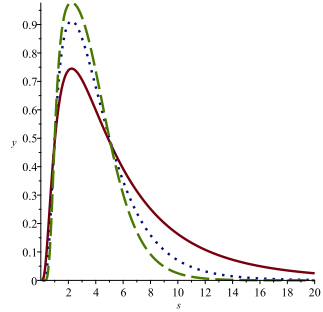}
\end{center}
\caption{\small Mellin-Gauss-Weierstrass kernel: $a=1$, $b=5$, $n=2$ (line), $n=3$ (point), $n=4$  (dash).}
\end{figure}
\end{example}

We now prove that formula  (\ref{q}) holds for the Mellin-Gauss-Weierstrass kernel.
\begin{proposition}\label{MGW}
For every $[a,b] \subset \mathbb{R}^+$ and 
$q \geq 1$ one has
\[s \mapsto \frac{n}{2\sqrt{\pi}} \int_a^b  e^{-\frac{n^2}{4} \ln^2 (\frac{t}{s})} \dfrac{dt}{t} \in L^q(\mathbb{R}^+), \quad  s \in \mathbb{R}^+\] 
for  $n \in \mathbb{N}$ large enough, depending on $q$.
\end{proposition}
\begin{proof}
Let $[a,b] \subset \mathbb{R}^+$. For every $s \in \mathbb{R}^+$ it is
\begin{eqnarray*}\label{mgwtri}
\int_a^b \, \widetilde{L}_n \Bigl( \frac{t}{s} \Bigr) \frac{dt}{t}= 
\frac{1}{\sqrt{\pi}}
\int_{\frac{n}{2}\ln\bigl( \frac{a}{s} \bigr)}^{\frac{n}{2}\ln\bigl( \frac{b}{s} \bigr)}
e^{-w^2}\, dw  .
\end{eqnarray*}
Now we prove that 
\begin{eqnarray} \label{mialc1}
s \mapsto
\int_a^b \widetilde{L}_n \Bigl( \dfrac{t}{s} \Bigr) \dfrac{dt}{t} \in L^{q}([b \, e^2, + \infty[)
\end{eqnarray}
for each $q \geq 1$ and $n$ large enough, depending on $q$.\\
If $s \geq b \, e^2$, then 
$\ln\Bigl( \dfrac{b}{s} \bigr) \leq -2, $
and hence 
\[\frac{n}{2}\ln\Bigl( \frac{a}{s} \Bigr) < 
\frac{n}{2}\ln\Bigl( \frac{b}{s} \Bigr) \leq -1.\]
 Therefore
\begin{eqnarray*}\label{ew0}
\frac{1}{\sqrt{\pi}}
\int_{\frac{n}{2}\ln\bigl( \frac{a}{s} \bigr)}^{\frac{n}{2}\ln\bigl( \frac{b}{s} \bigr)}
e^{-w^2}\, dw  \leq \frac{1}{\sqrt{\pi}}
\int_{\frac{n}{2}\ln\bigl( \frac{a}{s} \bigr)}^{\frac{n}{2}\ln\bigl( \frac{b}{s} \bigr)}
e^{w}\, dw = \\ = \frac{1}{\sqrt{\pi}} \, \Bigl( 
\, \Bigl(\frac{b}{s} \Bigr)^{n/2} - \Bigl(\frac{a}{s} \Bigr)^{n/2}
\Bigr),
\end{eqnarray*}
getting (\ref{mialc1}). 
Now we claim that 
\begin{eqnarray} \label{mialc2}
s \mapsto
\int_a^b \widetilde{L}_n \Bigl( \dfrac{t}{s} \Bigr) \dfrac{dt}{t} \in L^{q}(]0, a \, e^{-2}])
\end{eqnarray}
for all $q \geq 1$ and for every $n$.
If $0 < s \leq a \, e^{-2}$, then 
$\ln\Bigl( \dfrac{a}{s} \Bigr) \geq 2, $
and thus 
\[
\frac{n}{2}\ln\Bigl( \frac{b}{s} \Bigr) >
\frac{n}{2}\ln\Bigl( \frac{a}{s} \Bigr) \geq 1.\] 
Hence,
\begin{eqnarray*}\label{ew222}
\frac{1}{\sqrt{\pi}}
\int_{\frac{n}{2}\ln\bigl( \frac{a}{s} \bigr)}^{\frac{n}{2}\ln\bigl( \frac{b}{s} \bigr)}
e^{-w^2}\, dw  \leq \frac{1}{\sqrt{\pi}}
\int_{\frac{n}{2}\ln\bigl( \frac{a}{s} \bigr)}^{\frac{n}{2}\ln\bigl( \frac{b}{s} \bigr)}
e^{-w}\, dw = \frac{1}{\sqrt{\pi}} \, \Bigl( 
\, \Bigl(\frac{s}{a} \Bigr)^{n/2} - \Bigl(\frac{s}{b} \Bigr)^{n/2}
\Bigr),
\end{eqnarray*}
obtaining (\ref{mialc2}). Furthermore, we have
\begin{eqnarray} \label{mialc3}
s \mapsto
\int_a^b \widetilde{L}_n \Bigl( \dfrac{t}{s} \Bigr) \dfrac{dt}{t} \in L^{q}([a \, e^{-2}, b \, e^2]) 
\end{eqnarray} for all $q \geq 1$ and for every $n \in \mathbb{N}$, since the function in
(\ref{mialc3}) is continuous on $([a \, e^{-2}, b \, e^2])$.
Thus, the assertion  follows from 
(\ref{mialc1}), (\ref{mialc2}) and (\ref{mialc3}).
\end{proof}

\vskip.4cm
The graphs below show how  the moment and the Mellin-Gauss-Weierstrass kernels
approximate the 
 function $f \in {\mathcal C}_c(\mathbb{R}^+) $,
$$f(t)= (t-1) \chi_{[1,2]} + (3-t) \chi_{[2,3]},$$
  for    $n=10,15:$
\begin{figure}[h!]
\includegraphics[scale=.34]{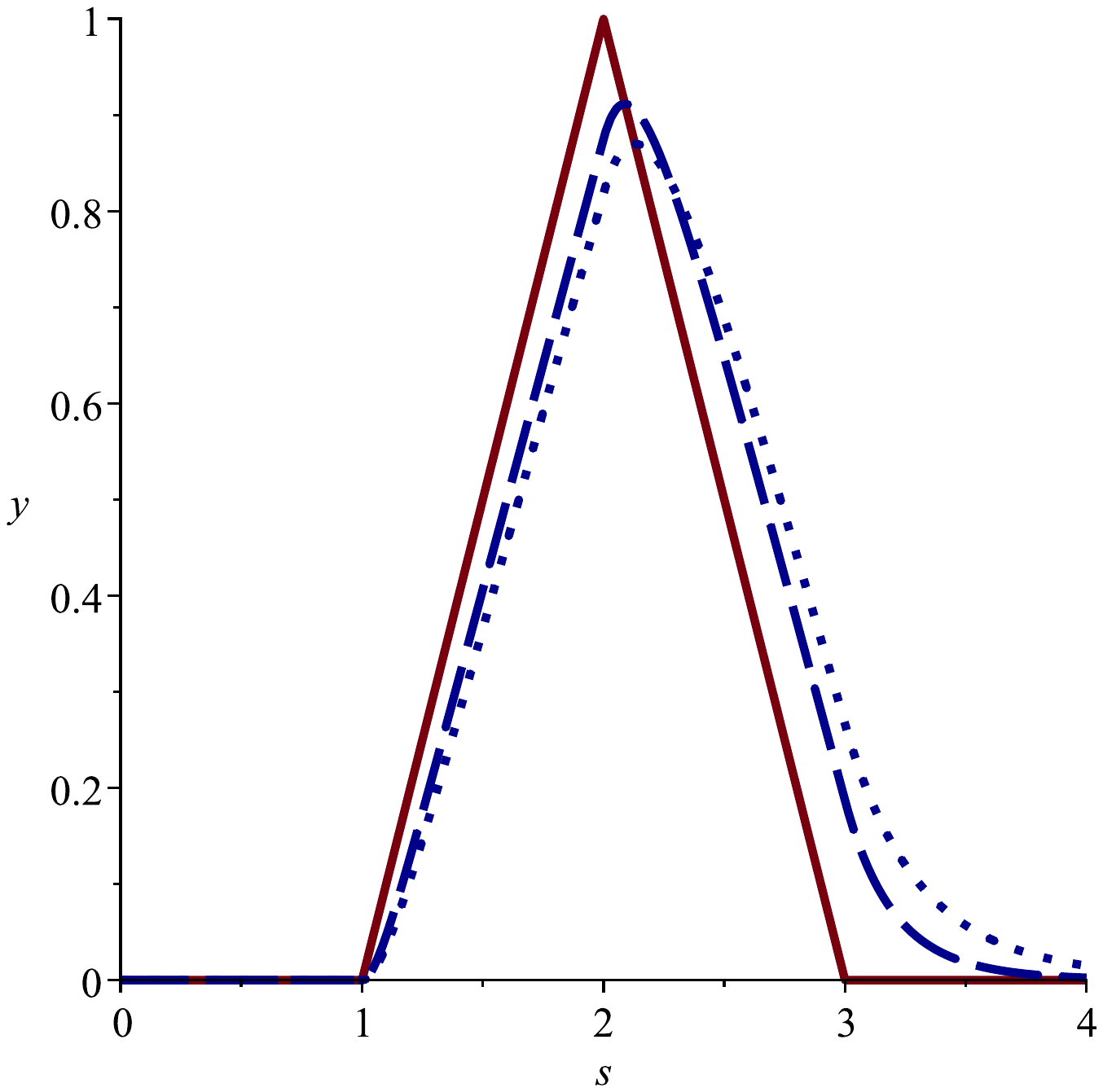}
\qquad 
\includegraphics[scale=.34]{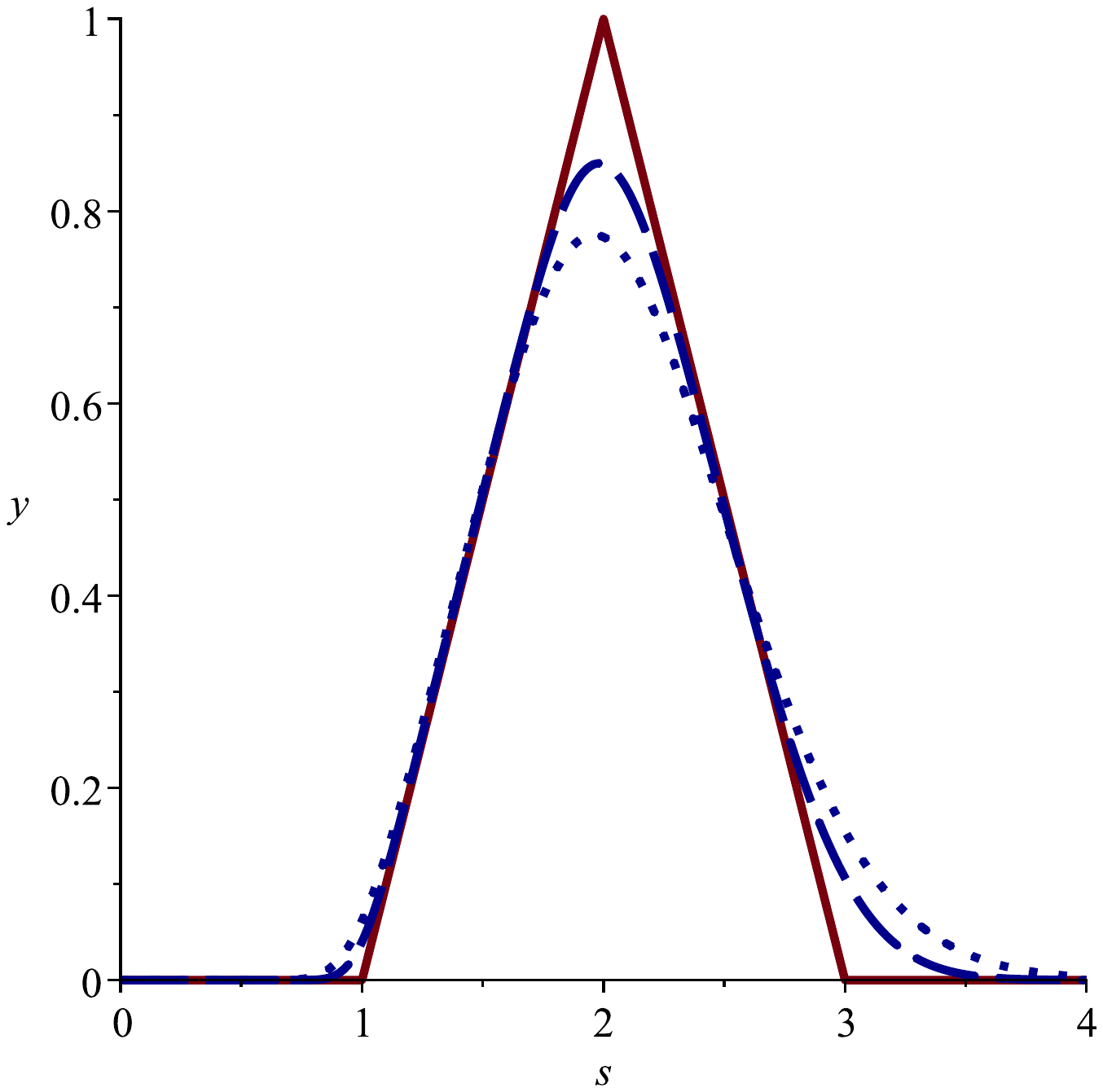}
\caption{\small Moment (Mellin-Gauss-Weierstrass) kernels applied to $f$:  $f$ (line), $\widetilde{T}_{10} f$ (point), $\widetilde{T}_{15} f$  (dash).}
\end{figure}
\\

Moreover we give the following
\begin{example}\label{exs3}
\emph{Mellin-Poisson-Cauchy kernel}.  \, 
Fix $p\in \mathbb{N}$, $p \geq 2$
and put, for every $n \in \mathbb{N}$, $t \in
\mathbb{R}^+$,
\begin{eqnarray*}\label{mpck}
\widetilde{L}_n(t)= C_p \, \dfrac{n}{(1+n^2 \ln^2 t)^p},  \,\,\, 
\mbox{ where  } \,\,\,C_p=\dfrac{2^{p-1}(p-1)!}{\pi
(2p-3)!!}.
\end{eqnarray*}

In \cite[formulas (5.10)-(5.14)]{BDMEDITERRANEAN} it is shown that 
\[\int_0^{+\infty} \widetilde{L}_n(t) \, \frac{dt}{t} =1
\quad \text{for  all   } n \in \mathbb{N}\]
and that, for any $\delta > 1$ and $n \in \mathbb{N}$,
it is
\begin{eqnarray}\label{mpc1}
\int_{\mathbb{R}^+ \setminus
[1/\delta, \delta]} \widetilde{L}_n(t) \, \frac{dt}{t}\leq 
C_p \, 
(\pi - 2 \arctan (n \, \ln \delta) ).
\end{eqnarray}
Thanks to de l'H\^{o}pital's rule, it is possible to show that 
\begin{eqnarray}\label{mpc2}
\lim_{n \to + \infty} 
\dfrac{\pi - 2 \arctan (n \, \,\ln \delta)}{1/n}=
\dfrac{2}{\ln \delta}.
\end{eqnarray}
From (\ref{mpc1}) and (\ref{mpc2}) 
 we find 
a positive real number $K_{\delta,p}$ and an integer 
$n_0=n_0(\delta, p)$ with 
\begin{eqnarray}\label{mpc3}
\int_{\mathbb{R}^+ \setminus
[1/\delta, \delta]} \widetilde{L}_n(t) \, \frac{dt}{t}\leq
K_{\delta,p} \, \, \dfrac{1}{n} 
\end{eqnarray}
whenever $n \geq n_0$. From (\ref{mpc3}) and Axioms \ref{convergenze}.c), \ref{convergenze}.d)
we deduce \ref{intro2mellin}.4).
\\
Moreover
$0 \leq \widetilde{L}_n(t) \leq n$ for each $n \in \mathbb{N}$
and $t \in \mathbb{R}^+$. From this and the 
$\mu$-integrability
 of  $\, \widetilde{L}_n$
we deduce that $(\widetilde{L}_n)_n \subset \mathcal M$.
Furthermore, we get 
\begin{eqnarray*}
s \mapsto & &
\int_a^b \widetilde{L_n} \Bigl( \frac{t}{s} \Bigr) \frac{dt}{t} 
=\frac{n 2^{p-1}(p-1)!}{\pi (2p-3)!!} \int_a^b \frac{1}{(1+n^2(\ln s - \ln t)^2)^p}
\frac{dt}{t} \\ & & \in L^q(\mathbb{R}^+), \quad  s \in \mathbb{R}^+
\end{eqnarray*}
for any $[a,b] \subset \mathbb{R}^+$, $q \geq 1$
and $n$ large enough, depending on $q$
(see \cite[formula (5.21)]{BDMEDITERRANEAN}),
getting relation (\ref{q}).
\begin{figure}[h!]
\includegraphics[scale=.34]{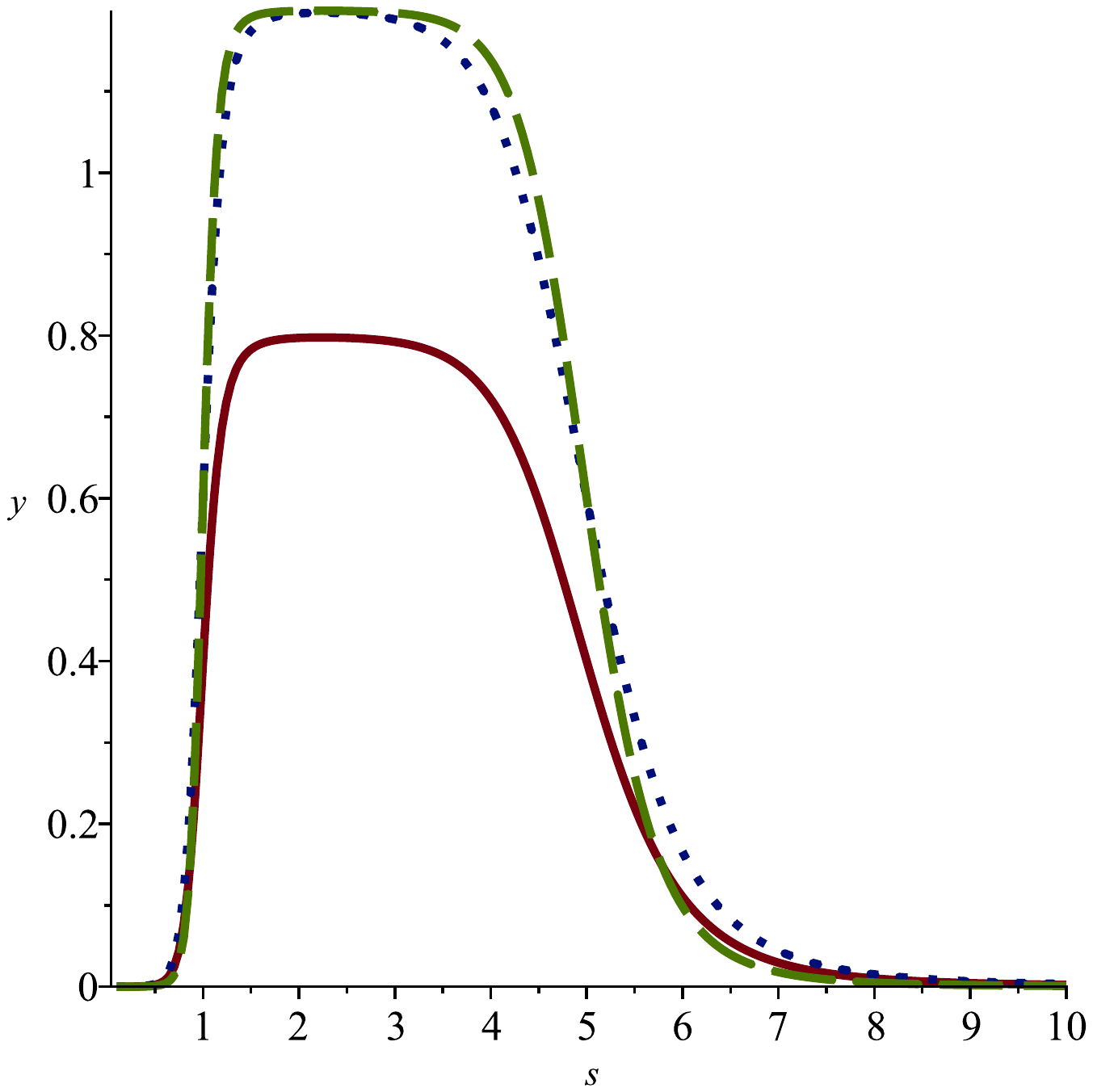} \quad
\includegraphics[scale=.34]{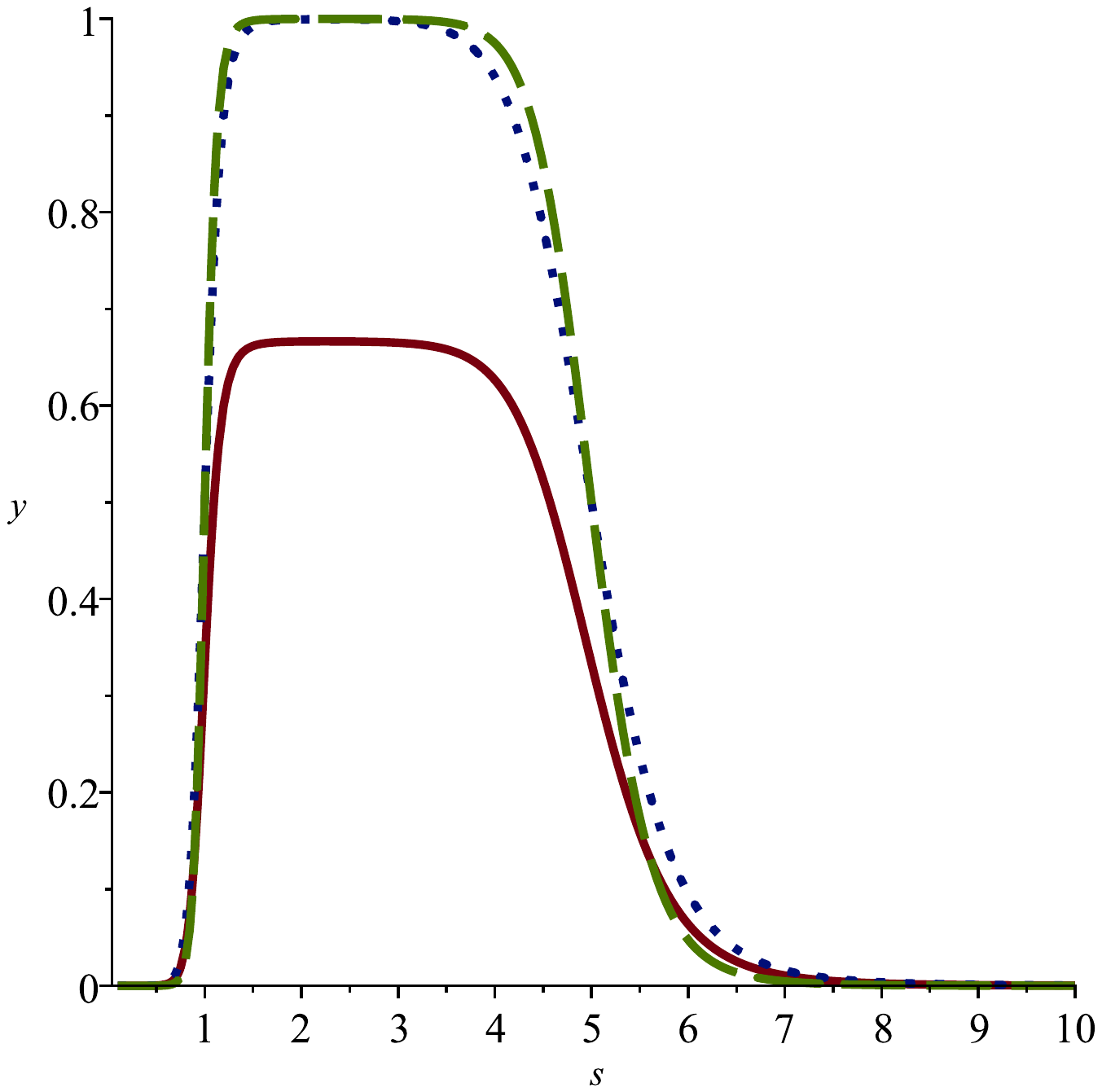}
\caption{\small Mellin-Poisson-Cauchy 
kernel: $a=1$, $b=5$, $n=2$ (line), $n=3$ (point), $n=4$  (dash), $p=3,4$.}
\end{figure}

\end{example}
\begin{example}\label{exs4}
 \emph{Linear operators}.  \, The simplest
example of functions $\Upsilon_n$ satisfying
\ref{intro3mellin}.2) is given by 
\[\Upsilon_n(u)=u, \qquad n \in \mathbb{N}, \quad  
u \in \mathbf{X}.\]
This case is widely studied in the literature 
when ${\mathbf{X}}=\mathbb{R}$ (see, e.g., \cite{BMV}
and the references therein).  
From \ref{intro3mellin}.1),
\ref{intro3mellin}.2) and Proposition 
\ref{intro3mellin} we obtain
\ref{intro2mellin}.5). Therefore, 
$\widetilde{\mathbb{K}}= (\widetilde{K}_n)_n$ is
$U$-singular, with $H=\mathbb{N}$ and $D^{(1)}=1$.
\\
Now observe that,
for every $n \in \mathbb{N}$,  
the function 
\begin{eqnarray}\label{integrabilityu}
\widetilde{K}_n(t,u)=\widetilde{L}_n(t) \, \, u, \quad
t \in \mathbb{R}^+, \, \, u \in \mathbf{X},
\end{eqnarray}
is $\mu$-integrable, 
since $\widetilde{L}_n$ is.
Moreover, we get that $\widetilde{K}_n(t,0) = 0 $ for each
$n \in \mathbb{N}$ and $t \in \mathbb{R}^+$. 
So, condition \,  \ref{assumptionsmellin}.b.1) holds.
Furthermore, it is straightforward to see
that  \, \ref{assumptionsmellin}.b.2)
is satisfied by taking $\psi_n(u)=u$ for all $n \in \mathbb{N}$ 
and $u \in {\mathbf{X}}^+$, since $(\widetilde{L}_n)_n \subset \mathcal M$. Thus, Assumptions \,\ref{assumptionsmellin}
are fulfilled.
\end{example}
\begin{example}\label{exs5}
 \emph{Nonlinear operators}.  \,
Let $0 \leq p \leq \infty, \,\, \mathbf{X}=L^p(
[0,1], \Sigma, \nu)$, where $\Sigma$ is the $\sigma$-algebra
of all measurable subsets of $[0,1]$ and $\nu$ is the 
Lebesgue measure on $[0,1]$. Note that such spaces
satisfy Axioms \ref{convergenze} and 
\ref{limsuppresentation}
are widely studied in Stochastic Integration, Stochastic Processes
and Brownian Motions, and it is possible to see that our integration
theory includes also some types of stochastic integration (see also
\cite{BCSVITALI, BS2021}).
For every $u \in {\mathbf X}, \,\, n \in \mathbb{N}$ and 
$t \in [0,1]$, set 
\begin{eqnarray*}\label{nonlinear}
\Upsilon_n(u(t))=\dfrac{n \, u(t) \, \lvert u(t) \rvert}{n \, \lvert u(t)\rvert +1}. 
\end{eqnarray*}
This case is an extension of some examples studied 
in the literature when ${\mathbf X}=\mathbb{R}$
(see \cite[Sections 5 and 6]{bmmellin}). 
First we note that, for every $n \in \mathbb{N}$, 
$\Upsilon_n$ is well-defined: indeed, for each
$u \in {\mathbf X}$ and $n \in \mathbb{N}$, we get 
\begin{eqnarray*}\label{upsilon0}
\lvert \Upsilon_n(u(t)) \rvert=\dfrac{n \, \, \lvert u(t)\rvert^2}{n \, \lvert u(t) \rvert +1}
\leq \lvert u(t) \rvert 
\end{eqnarray*} when $u(t) \neq 0$, and $\Upsilon_n (u(t))=0$
whenever $u(t)=0$. From this it follows that 
\begin{eqnarray}\label{upsilon1}
\lvert \Upsilon_n(u) \rvert \leq \lvert u \rvert
\end{eqnarray}
for any
$u \in {\mathbf X}$ and $n \in \mathbb{N}$,
and hence we deduce that $\Upsilon_n(u) \in 
{\mathbf X}$ whenever $u \in {\mathbf X}$. 

Now, observe that for each
$u \in {\mathbf X}$, $t \in [0,1]$ and $n \in \mathbb{N}$ it is 
\begin{eqnarray*}\label{upsilon2}
\lvert \Upsilon_n(u(t)) - u(t) \rvert=
 \bigg\lvert \dfrac{n \, u(t) \, \lvert u(t)\rvert }{n \, \lvert u(t)\rvert  +1} -u(t)  \bigg\rvert =
\dfrac{\lvert u(t) \rvert}{n \, \lvert u(t)\rvert +1} \leq \dfrac{1}{n} = 
\dfrac{1}{n} \cdot \mathbf{1}, 
\end{eqnarray*}
where $\mathbf{1}$
is the function which associates the real number $1$ to every 
$t \in [0,1]$. Since $\mathbf{1} \in \mathbf{X}$, 
then the $\Upsilon_n$'s satisfy
condition \ref{intro3mellin}.2) with 
$v=\mathbf{1}$ and $\sigma_n=\dfrac{1}{n}$, 
$n \in \mathbb{N}$.
From this, \ref{intro3mellin}.1) and Proposition 
\ref{intro3mellin} we obtain 
\ref{intro2mellin}.5). Hence, 
$\widetilde{\mathbb{K}}= (\widetilde{K}_n)_n$ is
$U$-singular, with $H=\mathbb{N}$ and $D^{(1)}=1$.
 \\
Moreover, from (\ref{upsilon1}) and the integrability 
of 
the functions $\widetilde{K}_n$ in (\ref{integrabilityu}),
arguing analogously as in \cite[Theorem 3.21]{dconv},
it follows that, for each $n \in \mathbb{N}$, the function
\begin{eqnarray*}\label{integrabilityupsilonn}
\widetilde{K}_n(t,u)=\widetilde{L}_n(t) \, \Upsilon_n(u),
\quad t \in \mathbb{R}^+, \, \, u \in \mathbf{X},
\end{eqnarray*}
is integrable with respect to $\mu$. Furthermore, 
since $\Upsilon_n(0)=0$, 
we obtain that $\widetilde{K}_n(t,0) = 0 $ for each
$n \in \mathbb{N}$ and $t \in \mathbb{R}^+$. 
Thus, condition \ref{assumptionsmellin}.b.1) is satisfied.
\\
Now, analogously as in \cite[Section 5, Example II]{bmmellin}, it is possible to see that
for each $u$,
$v \in {\mathbf X}$, $t \in [0,1]$ and $n \in \mathbb{N}$ one has
\begin{eqnarray*}\label{lip0}
\lvert \Upsilon_n(u(t)) - \Upsilon_n(v(t)) \rvert=
\bigg\lvert \dfrac{n \, u(t) \, \lvert u(t) \rvert }{n \, \lvert u(t)\rvert  +1} -
\dfrac{n \, v(t) \, \lvert v(t) \rvert }{n \, \lvert v(t) \rvert +1}  \bigg\rvert  \leq 
\lvert u(t)-v(t)\rvert, 
\end{eqnarray*}
and hence 
\begin{eqnarray*}\label{lip1}
\lvert \Upsilon_n(u) - \Upsilon_n(v) \rvert \leq \lvert u-v \rvert.
\end{eqnarray*}
So, condition \ref{assumptionsmellin}.b.2)
holds, with $\psi_n(u)=u$ for every $n \in \mathbb{N}$ 
and $u \in {\mathbf{X}}^+$, taking into account that 
$(\widetilde{L}_n)_n \subset \mathcal M$.
Hence, Assumptions \ref{assumptionsmellin}
are satisfied.
\end{example}


Then we deduce that 

\begin{corollary}\label{mellinexamples}
Let $q \in \mathbb{N}$, $\varphi(u)=|u|^q$ and
assume that $\widetilde{\mathbb{K}}=(
\widetilde{K}_n)_n$ is
defined by
\[ \widetilde{K}_n(t,u)=\widetilde{L}_n(t) \, 
\Upsilon_n(u), \quad 
n \in \mathbb{N}, \, \, t \in \mathbb{R}^+, \, \, u \in {\mathbf{X}}, \]
where $\widetilde{L}_n(t)$  is as in 
 Examples \ref{exs1} -- \ref{exs5}. 
Then, for any $f \in {\mathcal C}_c(\mathbb{R}^+)$,
the sequence $(\widetilde{T_n} f )_n$ 
is uniformly convergent to $f$  
on ${\mathbb{R}}^+$ and modularly convergent 
to $f$ with respect to $\rho^{\varphi}$, where the number $a$ obtained by the last convergence 
(in formula (\ref{modularconvergence0})) can be taken 
independently of $f$.
\begin{proof}
It is a consequence of 
 $U$-singularity, the argument in 
(\ref{aamu1}) and \cite[Theorems 6.5 and 6.9]{BS2021}.
\end{proof}
\end{corollary}
We would like to highlight that our theory 
 can be used also in the following contexts.
\begin{remarks}
\phantom{a}
\begin{itemize}
\item[a)]  In \ref{intro2mellin}.2), we consider
an infinite subset $H$ of natural numbers. This set can be
different from $\mathbb{N}$, and thus our theory 
includes filter convergence, which is in general 
strictly weaker than the classical one. \\
We remember, for the reader's simplicity, that a 
 \emph{filter} on $\mathbb{N}$ is  a family $\mathcal{F} \subset \mathcal{P}(\mathbb{N})$ such that $\emptyset \not \in {\mathcal F}$,
$A \cap B \in {\mathcal F}$ whenever $A$, $B \in \mathcal{F}$, and for any $A \in \mathcal{F}$ and $B \subset \mathbb{N}$ with
$A \subset B$, one has $B\in \mathcal{F}$.
We say that a filter on ${\mathbb N}$  is \textit{free} iff  it contains ${\mathcal F}_{\text{cofin}}$, which is  the filter of all cofinite subsets of  ${\mathbb N}$.
If ${\mathcal F}$
is any fixed free filter on ${\mathbb{N}}$ different from 
${\mathcal F}_{\textrm{cofin}}$, then it is advisable to 
take in \ref{intro2mellin}.2) an infinite set $H \in {\mathcal F}$,
such that $\mathbb{N} \setminus H$ is infinite 
(see also \cite{BS2021}).
In \cite{BDMEDITERRANEAN}
some other examples are presented, where the involved 
``Mellin-type'' kernels satisfy 
$U$-singularity conditions
with respect to ${\mathcal F}$-convergence, and not with respect to 
the usual convergence. 
\item[b)] Notice that our theory can include also
the multidimensional Mellin convolution operators, 
by considering $G=
{(\mathbb{R}^+)}^N$ endowed with the scalar product, 
where the distance $d_N$
on $G \times G$ is given by 
\[d_N(s,t)=\max\{d_{\ln}(s_1,t_1), d_{\ln}(s_2, t_2), \ldots, d_{\ln}(s_N, t_N)\},\]
and $\displaystyle{\mu(A)=\mu_N(A)=\int_A \frac{(dt)^N}
{\prod_{j=1}^N t_j}},
\quad A \in {\mathcal M}^N$,
where ${\mathcal M}^N$
is the family of all measurable subsets
of $({\mathbb{R}}^+)^N$ (see, e.g., \cite{AV2, Bu2003}).
Moreover, our theory can include also some types
of non-convolution operators (see, for instance, 
\cite[Example 3]{bm1}).
\item[c)] As in \cite[Example 3.3]{BMV}, it is possible to give some examples 
of nonlinear singular operators, in which the associated sequence
$(K_n)_n$ satisfies Lipschitz-type conditions with respect to more general functions
 $\psi_n$, $n \in \mathbb{N}$, as in
\ref{assumptionsmellin}.1).
\end{itemize}
\end{remarks}

\section{Conclusion}\label{sec13}
 In this article, by continuing  the investigation started in earlier papers, we have recalled an axiomatic theory of convergence, an abstract integral (with respect to possibly infinite finitely additive measures) and modulars in vector lattice setting.
We have given some examples of Mellin-type kernels and corresponding operators (moment, Mellin-Gauss-Weierstrass, Mellin-Poisson-Cauchy) and we have illustrated the behavior of the corresponding convergences by means of figures.


{\small 
\section*{Declarations}
\noindent {\bf Conflict of interests:} The authors declare no conflict of interest.\\
{\bf Author Contributions:}
All  authors  have  contributed  equally  to  this  work  for  writing,  review  and  editing. 
All authors have read and agreed to the published version of the manuscript.
\\
{\bf Avaibility of matherials and data:} The authors confirm that the data supporting 
the results of this study are available within the article [and/or] its supplementary material in the bibliography.\\
{\bf Funding:}
This research  was  supported by
 Ricerca di Base (2018) dell'Universit\`a degli Studi di Perugia "Metodi di Teoria dell'Ap\-prossimazione,  Analisi Reale, Analisi Nonlineare e loro Applicazioni";
 Ricerca di Base (2019) dell'Universit\`a degli Studi di Perugia-"Integrazione, Ap\-prossimazione,  Analisi Nonlineare e loro Ap\-plicazioni"
and
 "Metodi di approssimazione,
misure, Analisi Funzionale, Statistica e applicazioni alla ricostruzione di immagini e docu\-menti";
 "Metodi e processi innovativi per lo sviluppo di una banca di immagini mediche per fini diagnostici"
  funded by the Fondazione Cassa di Risparmio di Perugia (FCRP), (2018);
 "Metodiche di Imaging  non invasivo mediante angiografia OCT sequenziale per lo studio delle Retinopatie   degenerative dell'Anziano (M.I.R.A.)", funded by FCRP, (2019).

{\bf Acknowledgments}
The authors are also members of the  University of Perugia and of the  GNAMPA -- INDAM (Italy).
This research has been accomplished within the UMI Group TAA “Approximation Theory and Applications”.\\ 
}


\Addresses
\end{document}